\documentclass{article}

\usepackage{amsmath,leftindex,amsfonts,amssymb,amsthm,graphicx,float}

\newtheorem{theorem}{Theorem}[section]

\newtheorem{proposition}{Proposition}[section]
\newtheorem{remark}{Remark}[section]
\newtheorem{lemma}{Lemma}[section]
\newtheorem{corollary}{Corollary}[section]
\theoremstyle{definition}
\newtheorem{example}{Example}[section]

\renewcommand{\Re}{\operatorname{Re}}

\def\R{{\mathcal{R}}}

\def\A{{\mathcal{A}}}

\def\B{{\mathcal{B}}}
\def\M{\mathfrak{M}}
\def\K{\mathfrak{K}}

\def\CC{\mathbb{C}}
\def\RR{\mathbb{R}}

\def\HR{\mathbb{H}_r}
\def\RA1{\mathfrak{R}_{A'}}
\def\NA1{\mathfrak{N}_{A'}}

\def\id{\operatorname{id}}

\newcommand{\mat}[4]{\left[\begin{array}{cc}#1 & #2 \\ #3 & #4 \\
	\end{array}\right]}

\author{Bogdan D. Djordjevi\'c\footnote{Mathematical Institute of the Serbian Academy of Sciences and Arts, (MISASA), Serbia. \\
Institute of
Mathematics and Informatics, Bulgarian Academy of Sciences (IMI-BAS), Bulgaria.\\
Faculty of Mathematics and Informatics, Sofia University ``St. Kliment Ohridski" (FMI-SU), Bulgaria.\\
bogdan.djordjevic@turing.mi.sanu.ac.rs; bogdan.djordjevic93@gmail.com
}}
\date{}

\begin{document}

\title{Solving the Sylvester equation in Banach modules}

\maketitle

\begin{abstract} 
For given unital complex Banach algebras $\A_1$ and $\A_2$, let $\M$ be a Banach module acting between them. Let $a\in\A_1$, $b\in\A_2$, and $c\in\M$ be provided such that $\sigma_{\A_1}(a)\cap\sigma_{\A_2}(b)\neq\emptyset$. In this paper we completely characterize the consistency of the Sylvester equation $$ax-xb=c.$$
Precisely, we establish verifiable sufficient and necessary solvability conditions, and we provide some formulas for particular solutions $x\in\M$ when the equation is solvable.  Moreover, we characterize the uniqueness of the solutions. \end{abstract}


{ Keyphrases: Sylvester equation; Banach modules; Banach algebras; Operator matrices.}

\section{Introduction}
Let $\A_1$ and $\A_2$ be unital complex Banach algebras. A Banach space $\M$ defines a Banach $(\A_1,\A_2)-$module, if  it is simultaneously a left Banach module for $\A_1$, and a right Banach module for $\A_2$, with respect to some (fixed) multiplications $\cdot_1:\A_1\times\M\to\M$, and $\cdot_2:\M\times\A_2\to\M$. It is inherently understood that the sub-multiplicativity is preserved within such structures, that is,
$$\|a_1\cdot_{1} c\cdot_{2} a_2\|_{\M}\leq \|a_1\|_{\A_1}\cdot\|c\|_\M\cdot\|a_2\|_{\A_2},$$
for any $a_i\in\A_i$, $c\in\M$, and that the associative law holds:
$$\left(a_1\cdot_1 c\right)\cdot_2 a_2=a_1\cdot_1 \left(c\cdot_2 a_2\right)\in\M.$$
It is also assumed that $\M$ is not a trivial module, meaning that, if $c\neq 0_\M$, then $1_{\A_1}\cdot c\neq 0_\M$ and $c\cdot 1_{\A_2}\neq 0_\M$.

With respect to the previous notation, let $a\in\A_1$, $b\in\A_2$, and $c\in\M$ be given elements. In this paper we study and solve the algebraic Sylvester equation  
\begin{equation}
\label{sylm-}
ax-xb=c
\end{equation}
for the unknown solution(s) $x\in\M$: we establish both sufficient and necessary solvability conditions, and provide several formulas for particular solutions. In addition, we obtain new necessary and sufficient conditions under which the solution is unique. To the best of our knowledge, these results are original and represent a significant progress compared to the existing findings in the available literature.

\subsection{Sylvester equations: Known results}

The Sylvester equation \eqref{sylm-} is a well-known equation which has been extensively studied in the matrix and operator setting (see \cite{R2}, \cite{RBPR}, \cite{RBMU}, \cite{MD}, \cite{SGL}, \cite{RW}, \cite{JJS}, \cite{HWXSJH}), and even in rings with involution \cite{MGLD}, Hilbert $C^*-$modules  \cite{ZMREMMFM}, and semisimple Banach algebras \cite{AS}.

The most well-known criterion regarding the existence and uniqueness of a solution lies in spectral properties of $a$ and $b$: when $a$ and $b$ are bounded linear operators given in the corresponding Banach spaces, provided such that their spectrums are disjoint, then for any given bounded linear operator $c$, there exists a unique bounded linear operator $x$ (in the respective operator space), such that \eqref{sylm-} holds. Moreover, when $a$, $b$, and $c$ are matrices (or operators in finite-dimensional vector spaces), then there exists a unique solution $x$ to \eqref{sylm-}, if and only if the spectra of $a$ and $b$ are disjoint. For this reason, in the bounded operator (and matrix) setting, the equation is said to be \emph{regular} whenever the spectrums of $a$ and $b$ are disjoint, and it is called \emph{singular} otherwise.  Regular Sylvester equations and their generalizations in matrices and operators have been extensively studied in the available literature, see e.g.  \cite{R2}, \cite{RBPR}, \cite{RBMU}, \cite{MD}, \cite{RW}, \cite{JJS}, \cite{HWXSJH}, and numerous references therein.

 Conversely, singular Sylvester equations lack such results. There are some special cases in which the equation has been solved successfully: for instance, when $\A_1$, $\A_2$, and $\M$ are finite-dimensional, i.e., when the equation \eqref{sylm-} is provided in matrices, solvability conditions have been established and the solution set has been described completely,  see \cite{ECLHDS}, \cite{DIN}, \cite{BDND1}, \cite{FRG}, \cite{GW}, \cite{REH}, and \cite{ASMD}. Moreover, when $\A_1$, $\A_2$, and $\M$ are Banach spaces of bounded linear operators (or even closed operators in some instances), papers \cite{BL}, \cite{BD2}, \cite{BDD2},  \cite{ZMREMMFM}, and \cite{HWXSJH}, provide some sufficient solvability conditions for \eqref{sylm-}, as well as some forms for particular solutions in those settings. These results usually rely on the existence of Riesz points in the respective spectra, or demand some stronger properties, like self-adjointness, Drazin-Koliha invertibility, von Neumann regularity, etc.  The difficulty with the singular equation is that the nonempty spectral intersection for $a$ and $b$ does not guarantee (nor disprove) neither the existence nor the uniqueness of solutions, therefore this scenario always has some underlying issues. Moreover, the solutions obtained in the said papers (\cite{BD1}, \cite{BD2}, \cite{BDD2},  \cite{ZMREMMFM}, and \cite{HWXSJH}) rely heavily on the spectral properties of $a$ and $b$ and the corresponding eigenspaces, thus making them not easily obtainable. 

Apart from the previously mentioned results, papers \cite{BD1}, \cite{BDZG}, and \cite{HC} offer some interesting necessary solvability conditions for \eqref{sylm-}, without an explicit form for a particular solution. Finally, the most well-known necessary sovlability criterion for \eqref{sylm-} (sans uniqueness) is the Roth's removal rule, which states that, if \eqref{sylm-} is solvable, then the matrices
$\mat{a}{c}{0}{b}$ and $\mat{a}{0}{0}{b}$ are similar, and the converse statement holds only in finite-dimensional spaces. Still, none of the necessary conditions mentioned here can be utilized for obtaining sufficient sovlability conditions for \eqref{sylm-} in the infinite-dimensional case, nor do they offer a way for finding a particular solution.

\subsection{Banach module setting}

Throughout this paper we are working with fixed unital complex Banach algebras $\A_1$ and $\A_2$, and a fixed Banach  $(\A_1,\A_2)-$module $\M$. We denote this convention as $(\A_1,\A_2,\M)$. For easier notation, we simplify the indexing and write
$$a_1\cdot_1c\cdot_2a_2\equiv a_1ca_2,\quad a_i\cdot_{\A_i}b_i\equiv a_ib_i,\quad c\in\M,\  a_i,b_i\in\A_i,\quad i=1,2.$$
Similarly, we write $\|\cdot\|$ instead of $\|\cdot\|_{\M}$, or $\|\cdot\|_{\A_i}$, etc. The spectrum of $a_i$ within the respective algebra $\A_i$, $i=1,2$, will be written as $\sigma_{\A_i}(a_i)$, or as $\sigma(a_i)$ if there is no chance of confusion.   The algebra of bounded linear operators over $\M$ is denoted as $\B(\M)$. For a given $L\in\B(\M)$, the image (range) of the operator $L$ is denoted as $\R(L)$. Finally, for any $a_i\in\A_i$, the set $(a_i)'$ denotes the set of elements in $\A_i$ which commute with $a_i$. For more on Banach (and Hilbert $C^*-$) modules consult  \cite{Trap},  \cite{BDI}, \cite{BIT}, \cite{Bel}, \cite{BDZG}, \cite{ZMREMMFM}, and numerous references therein.\\

Accordingly, we assume that $a\in\A_1$, $b\in\A_2$, and $c\in\M$ from \eqref{sylm-} are also provided and fixed, and we focus on inspecting the consistency of \eqref{sylm-}. It is straightforward to see that, even in the Banach module setting $(\A_1,\A_2,\M)$, if $a$ and $b$ have disjoint spectrums then there always exists a unique solution $x\in\M$ corresponding to the given $c\in\M$. Indeed, by introducing the bounded multiplication operators in $\M$ $Ax:=ax$ and $Bx:=xb$, it follows that $\sigma_{\B(\M)}(A)\subset\sigma_{\A_1}(a)$ and  $\sigma_{\B(\M)}(B)\subset\sigma_{\A_2}(b)$, ergo $$\sigma(a)\cap\sigma(b)=\emptyset\Rightarrow \sigma_{\B(\M)}(A)\cap\sigma_{\B(\M)}(B)=\emptyset.$$
Moreover, the operators $A$ and $B$ commute, and can be embedded into the same maximal commutative Banach subalgebra $\B$ of $\B(\M)$ that contains them. Furthermore, the said commutative algebra $\B$ preserves their spectrums. Then, due to the Gel'fand transform (see \cite{RBPR}, \cite{VM}, and \cite{AS}), it follows that
$$0\notin\sigma_{\B(\M)}(A)-\sigma_{\B(\M)}(B)\supset\sigma_{\B}(A-B)=\sigma_{\B(\M)}(A-B).$$
In other words, if $\sigma(a)\cap\sigma(b)=\emptyset$, then for every $c\in\M$ there exists a unique $x\in\M$ such that
$$(A-B)x=c\Leftrightarrow ax-xb=c,$$
i.e., the equation \eqref{sylm-} is \emph{regular}. Keeping the notation consistent with \cite{ECLHDS}, \cite{BD1}, \cite{BDD2}, \cite{BDND1},  and \cite{ASMD}, when $\sigma(a)\cap\sigma(b)\neq\emptyset$, the equation \eqref{sylm-} is said to be \emph{singular}.

Due to the previous discussion, in this paper we focus only on the case when $\sigma(a)\cap\sigma(b)\neq\emptyset$. As opposed to the previously mentioned papers that treat the singular equation, the approach in this paper is completely different: rather than conducting a rigorous spectral analysis for the elements $a$ and $b$, we introduce and study a matrix algebra $\mathcal{M}$ which is generated by $a$, $b$, and $c$. This way, the existence of a particular solution, as well as its uniqueness, are closely related to the existence of a nonprimary square root (and its modifications) of a certain matrix which belongs to the said matrix algebra  $\mathcal{M}.$ Though this idea resembles the Roth's removal rule, the advantage of our matrix algebra lies in the fact that we keep the ``if and only if'' statements, rather than the ``only if'' one. Moreover, Roth's removal theorem does not make much sense when $c=0$, whereas  our results work with an ease in this case (see Section \ref{matrixalgebra}).  By employing the said construction, we reduce the problem to a simple  (and verifiable) consistency criterion for a two-equation system, see Theorem \ref{upmv} at the end of the paper. \\

Throughout this paper, we assume that $(\A_1,\A_2,\M)$ is fixed and provided as above.  Moreover, it is understood that the elements $a\in\A_1$, $b\in\A_2$, and $c\in\M$ are provided as well, such that $\sigma(a)\cap\sigma(b)\neq\emptyset$. By studying this two-sided module setting, we simultaneously enclose the problem in several different types of classical operator algebras and modules: bounded, compact, (semi-)Fredholm, finite-rank, with finite or semi-finite traces, and even some more advanced structures, like the unital quasi *-algebras and non-commutative $\ell_p$ spaces, see e. g. \cite{Trap}, \cite{BDI}, \cite{BIT}, \cite{Bel}, \cite{JB}, \cite{BDZG},  \cite{DS}, \cite{VM}.

\section{Transferring the problem}\label{transfer}

In this section we rephrase the initial problem in the following way. By the spectral mapping theorem, the equation \eqref{sylm-} is mathematically equivalent to
\begin{equation}
\label{sylR}
(a+\lambda 1_{\A_1})x-x(b+\lambda 1_{\A_2})=c,
\end{equation}
for an arbitrary scalar $\lambda$. Thus for any given  angle $\alpha\in(0,\pi/2)$ we can choose $\lambda$ to be a sufficiently large positive number, such that $\sigma(a+\lambda)\cup\sigma(b+\lambda)$ is contained in the sector
$\{z:0\neq z, -\alpha,<\arg z<\alpha\}$. For convenience, for a given $\alpha\in(0,\pi/2)$, we denote that sector as
$$\Lambda_\alpha:=\{z:0\neq z, -\alpha,<\arg z<\alpha\}.$$
When necessary, we will assume that $\sigma(a)\cup\sigma(b)\subset\Lambda_\alpha$, without the loss of generality. Moreover, by $\HR$ we denote the open right complex half-plane with positive real parts:
$$\HR=\{z\in\CC: \Re z>0\}=\bigcup_{0<\alpha<\pi/2}\Lambda_\alpha.$$

\subsection{The regular equation $a\bar{c}+\bar{c}b=c$}

By transforming the equation \eqref{sylm-} into \eqref{sylR} above,  the following analysis immediately follows:

\begin{lemma}\label{expa} Let $a\in\A_1$ be such that $\sigma(a)\in\HR$. Then the function $t\mapsto e^{-ta}$, $t\in[0,+\infty)$, vanishes when $t\to+\infty$.
\end{lemma}
\begin{proof}
By the Cauchy integral formula, it follows that
$$e^{-ta}=\frac{1}{2 \pi i}\int_{\gamma} e^{-tz}(z-a)^{-1}dz$$
where $\gamma$ is a contour surrounding $\sigma(a)$, and lies entirely in $\HR $ as well. Then the Cauchy-Schwartz inequality for the dual norms $\|\cdot\|_1$ and $\|\cdot\|_\infty$ gives
$$\|e^{-ta}\|\leq\frac{1}{2\pi}\int_{\gamma}\|(z-a)^{-1}\||dz|\cdot \sup_{z\in\gamma}\left|e^{-tz}\right|.$$
The function 
$$z\mapsto \frac{1}{2\pi}\int_{\gamma}\|(z-a)^{-1}\||dz|$$
is continuous on the compact $\gamma$, therefore, it is uniformly continuous and thus it is bounded by some $C_\gamma>0$. On the other hand, due to $\left|e^{z}\right|=e^{\Re z}$, and $\Re z>0$ for every $z\in\gamma$, we have
$$\|e^{-ta}\|\leq C_\gamma \cdot \max_{z\in \gamma}e^{-t \Re z},$$
which vanishes when $t\to+\infty$.
\end{proof}

\begin{lemma}\label{regeq} Let $a\in\A_1$ and $b\in\A_2$ such that $\sigma(a)\cup\sigma(b)\subset\HR$. Let $c\in\M$ be arbitrary. Then the regular Sylvester equation
\begin{equation}
\label{sylreg}
a\bar{c}+\bar{c}b=c
\end{equation}
has the unique solution $\bar{c}\in\M$ given as
\begin{equation}
\label{sylsolm+}
\bar{c}=\int_0^\infty e^{-ta}ce^{-tb}dt. 
\end{equation}
\end{lemma}
\begin{proof} Since $\sigma(a)\cap\sigma(-b)=\emptyset$, by the previous discussion it follows that the equation \eqref{sylreg} is indeed regular and with a unique solution $\bar{c}$ for a given $c$. On the other hand, it is obvious that the integral in \eqref{sylsolm+} is absolutely convergent, therefore it can be treated as a Bochner integral with respect to the standard scalar Lebesgue measure on $[0,+\infty)$, and as such it subjects to the fundamental theorem for the Bochner integral: in other words, we have 
$$
\begin{aligned}
&a\int_0^\infty e^{-ta}ce^{-tb}dt+\int_0^\infty e^{-ta}ce^{-tb}dt\ b=\\
=&-\int_0^\infty\left( e^{-ta}ce^{-tb}\right)'_t dt=-\left.\left(e^{-ta}ce^{-tb}\right)\right|_{t=0}^{t\to+\infty}=\\
=&c-\lim_{t\to+\infty}e^{-ta}ce^{-tb}=c.
\end{aligned}$$
\end{proof}
The previous result also follows from the multiplication operator analysis: let $A, B\in\B(\M)$ be the appropriate multiplication operators in $\M$, $Ay=ay$ and $By=yb$. Then $A$ and $B$ commute in $\B(\M)$, and  $$\sigma_{\B(\M)}(A)\cup \sigma_{\B(\M)}(B)\subset\sigma_{\A_1}(a)\cup\sigma_{\A_2}(b)\subset\HR.$$
It follows by the virtue of the Gel'fand transform for commutative Banach algebras that 
$$\int_0^\infty e^{-ta}ye^{-tb}dt=\int_0^\infty e^{-tA}e^{-tB}ydt=\int_0^\infty e^{-t(A+B)}dt\ (y)=(A+B)^{-1}y,\quad y\in\M.$$
Therefore, the expression  $e^{-t(A+B)}$ is well-defined in $\B(\M)$, and so is the integral $\int\limits_0^\infty e^{-t(A+B)}dt=(A+B)^{-1}$, by Lemma \ref{expa}. 

\subsection{The equation $az=wb$}

In addition to obtaining the consistency conditions for \eqref{sylm-}, it is important to inspect the uniqueness of the solution. Since the problem is linear, the uniqueness of the solution to \eqref{sylm-} is equivalent to the uniqueness of the solution to the corresponding homogeneous equation (also known as the intertwining problem, see e.g. \cite{JB} and \cite{HM}):
\begin{equation}
\label{sylhom} ax=xb.
\end{equation} 
In other words, if $ax=xb\Rightarrow x=0_\M$, then there is at most one solution to $ax+xb=c$. Indeed, if there were at least two different solutions, say $x_1$ and $x_2$, then $x_1-x_2$ would solve \eqref{sylhom}, which is impossible. Therefore, it is significant to study the equation \eqref{sylhom} as well. As opposed to the matrix case (see \cite{MD}), it is quite simple to construct a singular homogeneous Sylvester equation with the unique trivial solution:

\begin{example} Let $\A_1$ be a noncommutative unital Banach algebra, and let $\A_2$ be any nontrivial unital Banach algebra. Let $\M$ be any nontrivial Banach $(\A_1,\A_2)-$module, i.e., 
$$1_{\A_1}\cdot m=0_\M\Rightarrow m=0_\M,\quad m\cdot 1_{\A_2}=0_\M\Rightarrow m=0_\M.$$ 
For a nonzero scalar $\lambda$, let $b=\lambda 1_{\A_2}$, and let $a$ be any element in $\A_1$ such that the linear pencil $a_\lambda:=a-\lambda1_{\A_1}$ is singular in $\A_1$ (say, let it be a right zero divisor), but in the manner that $a_\lambda$ is not a left zero divisor. Then $\sigma(a)\cap\sigma(b)=\{\lambda\}$ and
$$ax=xb\Leftrightarrow a_\lambda x=0\Rightarrow x=0.$$
\hfill$\clubsuit$
\end{example}
Surprisingly enough, obtaining a particular solution to \eqref{sylR} (equivalently, to \eqref{sylm-}), can be transferred to solving a generalized form of \eqref{sylhom} via the following lemma: 

\begin{lemma}
\label{C1C2Y}
Let $a$ and $b$ be invertible in $\A_1$ and $\A_2$, respectively. The following statements are equivalent:
\begin{itemize}
\item[(a)] There exists a particular solution $x\in\M$ to \eqref{sylm-}.
\item[(b)] There exist $c_1$, $c_2,$ and $y\in\M$, such that $c=c_1-c_2$ and $(c_1+y)b=a(c_2+y)$.
\end{itemize}  
In that case, a particular solution can be chosen to be 
\begin{equation}\label{solparinv}
x:=a^{-1}(c_1+y)\equiv (c_2+y)b^{-1}.\end{equation}
\end{lemma}
\begin{proof}
$(a)\Rightarrow(b):$ If $x$ is a particular solution, then $c_1:=ax$, $c_2:=xb$, and $y:=0$ (or any solution to the homogeneous equation \eqref{sylhom}). Then $c=ax-xb=c_1-c_2$ and
$$(c_1+y)b=axb+yb=axb+ay=a(c_2+y).$$
$(b)\Rightarrow(a):$ Conversely, by choosing $x$ as in \eqref{solparinv}, it directly follows that it is a solution to \eqref{sylm-}.
\end{proof}
\begin{remark} From $c=c_1-c_2$, it follows that $c_1$ and $c_2$ determine each other uniquely (if one is known then so is the other). On the other hand, the element $y$ is any solution to the homogeneous Sylvester equation \eqref{sylhom}. Given that we seek any particular solution in this lemma, it suffices to choose $y=0$ without the loss of generality. \end{remark}
In order to find a particular solution in the form \eqref{solparinv}, one ought to address the generalized intertwining problem (or the generalized homogeneous Sylvester equation)
\begin{equation}\label{int} az=wb 
\end{equation}
for the unknown pair $(z,w)\in\M\times \M$. The elements $z$ and $w$ are uniquely determined by one another, via the relation $z=a^{-1}wb$. Therefore we are going to analyze the set
\begin{equation}
\label{Omega}\Omega=\{(z,w): az=wb\}.
\end{equation}
Note that its ``diagonal''  is precisely the solution set for the homogeneous Sylvester equation \eqref{sylhom}:
\begin{equation}
\label{Omega0}\Omega_0:=\{(z,z): az=zb\}=\{(z,w)\in\Omega:w-z=0\}.\end{equation}
Moreover, the set $\Omega\setminus\Omega_0$ is nonempty, because for any chosen $z$ there exists (a unique) $w$, given as $w=azb^{-1}$, such that \eqref{int} holds. Within the set $\Omega$, due to Lemma \ref{C1C2Y}, we are interested in the existence of the set
\begin{equation}
\label{OmegaC}\Omega_c=\{(z,w)\in\Omega: w-z=c\}.\end{equation}
\begin{corollary}
With respect to the previous notation, there exists a particular solution to \eqref{sylm-} if and only if $\Omega_c\neq\emptyset$.
\end{corollary}

\section{The matrix algebra $\mathcal{M}$}\label{matrixalgebra}

For the fixed $(\A_1,\A_2,\M)$, we introduce the following algebraic construction: we assume there exists a complex Banach space $\K$, which is  a Banach $(\A_2,\A_1)-$module (note the reverse order of $\A_1$ and $\A_2$). This is always possible: for any given Banach space $\mathfrak{B}$, choose $\K:=\{0_\mathfrak{B}\}$, and define the respective multiplications as $a_2\cdot0_\mathfrak{B}=0_\mathfrak{B}\cdot a_1=0_\mathfrak{B}$, for every $a_i\in\A_i$, $i=1,2$. In that case, the module $\K$ is said to be \emph{trivial}. This, of course, need not be the only choice for $\K$: it can be provided in a manner that it is a proper nontrivial Banach space. However, it will soon be shown that the choice of $\K$ does not affect our calculations: we can always use just the zero vector from any $\K$ without the loss of the generality.

Nonetheless, the space $\A_1\times \M\times \K\times \A_2$ as such exists, and can be embedded into a $2\times 2$ matrix algebra $\mathcal{M}$, in the way that
$$(a,m,k,b)\mapsto\mat{a}{m}{k}{b}\in\mathcal{M},$$
for every $(a,m,k,b)\in \A_1\times \M\times \K\times \A_2$. Precisely, the algebra $\mathcal{M}$ is equipped with the matrix multiplication $\circ$, defined as
$$\mat{a_1}{m_1}{k_1}{b_1}\circ\mat{a_2}{m_2}{k_2}{b_2}:=\mat{a_1a_2}{a_1m_2+m_1b_2}{k_1a_2+b_1k_2}{b_1b_2},$$
for $(a_i,m_i,k_i,b_i)\in \A_1\times \M\times \K\times \A_2$,  $i=1,2$.  The ordered pair $\left(\mathcal{M},\circ\right)$ represents a semigroup with the unit $1_{\mathcal{M}}:=\mat{1_{\A_1}}{0_{\M}}{0_{\K}}{1_{\A_2}}$, i.e., it is a monoid. When both $a$ and $b$ are invertible in the respective algebras, the inverse of the matrix $\mat{a}{m}{k}{b}$ is given as $\mat{a^{-1}}{-a^{-1}mb^{-1}}{-b^{-1}ka^{-1}}{b^{-1}}$.The set of invertible matrices in $\mathcal{M}$ is denoted as $\mathcal{M}^{-1}$. As before, we will often omit the notation $\circ$ when the matrix multiplication is obvious.

\subsection{(Non)-primary square roots}

Let $a\in\A_1$ and $b\in\A_2$ be given elements in the respective algebras. Denote by $N_0^\pm$ the matrices
\begin{equation}\label{N0}
N^{\pm}_0:=\mat{a}{0}{0}{\pm b},
\end{equation}
and by $N_0^2$:
\begin{equation}
\label{M}N_0^2:=\mat{a^2}{0}{0}{b^2}.\end{equation}
In that sense, there exists the set of square roots of $N_0^2$ in $\mathcal{M}$:
\begin{equation}\label{sqrtM}\sqrt{N_0^2}=\{N\in\mathcal{M}: N^2=N_0^2\}\neq\emptyset,\end{equation}
where $N^2=N\circ N$. 
Since the sets $\sqrt{a^2}$ and $\sqrt{b^2}$ are nonempty, where
$$\sqrt{a^2}=\{\xi\in\A_1: \xi^2=a^2\},\ \sqrt{b^2}=\{\eta\in\A_2: \eta^2=b^2\},$$
we say that $N$ is a \emph{primary} square root of $N_0^2$ if $N$ is contained in the set
$$\left\{\mat{\xi}{0}{0}{\eta}:\  \xi\in\sqrt{a^2}, \ \eta\in\sqrt{b^2}\right\},$$
and is \emph{nonprimary} otherwise. We define the set:
\begin{equation}
\label{sqrtM'}\left(\sqrt{N_0^2}\right)^{\sim}_{-}:=\left\{N\in\sqrt{N_0^2}: N\textrm{ is similar to }N_0^-\right\},
\end{equation}
with ``similarity'' being understood in the sense that there exists an invertible matrix $U\in\mathcal{M}^{-1}$ such that 
$N=U\circ N_0^- \circ U^{-1}$. Since it contains $N^{-}_0$, the set $\left(\sqrt{N_0^2}\right)^{\sim}_{-}$ is nonempty.

\subsection{Coupled homogeneous equations}

For given elements $a\in\A_1$ and $b\in\A_2$, assume the matrices $N_0^\pm$ and $N_0^2$ are defined as \eqref{N0} and \eqref{M}, respectively. Since these block-diagonal matrices are symmetric with respect to the entities $a$ and $b$, it is convenient to, in addition to the equation \eqref{sylhom}, consider the adjoint equation in $\K$:
 \begin{equation}\label{sylhomadj}
 by=ya
 \end{equation}
 for the unknown $y\in\K$. The equation \eqref{sylhomadj} is solvable for $y=0_\mathfrak{K}$, regardless of whether $\K$ is trivial or not. Therefore, we are going to study the coupled system of the homogeneous equations
  \begin{equation}
  \label{couplhom}
  \begin{cases} ax=xb,\\by=yb
  \end{cases}
  \end{equation}
  and we are going to inspect the existence of the solution $(x,y)\neq(0_\M,0_\K)$. As opposed to the matrix case (see e.g. \cite{MD}), it is quite possible to have a system \eqref{couplhom}, in which one of the equations has only the trivial solution, while the other one allows nontrivial solutions.

\begin{lemma}\label{homsylsqrtM} Let $a$ and $b$ be invertible in their respective algebras, such that $\sigma(a)\cap\sigma(b)\neq\emptyset$, and $\sigma(a)\cup\sigma(b)\subset\HR$. Then:
\begin{itemize}
\item[(a)] Every solution $x$ to \eqref{sylhom} (resp. every solution $y$ to \eqref{sylhomadj}) defines a similarity matrix $U=\mat{u_1}{u_2}{u_3}{u_4}$, such that $u_2$ solves \eqref{sylhom}, $u_3$ solves \eqref{sylhomadj}, and $U^{-1}N^-_0U\in\left(\sqrt{N_0^2}\right)^{\sim}_{-}$. Moreover, if $(x)\neq 0_\M$, (resp. if $y\neq 0_\K$), then the said similarity matrix $U$ can be chosen such that $U^{-1}N^-_0U\in\left(\sqrt{N_0^2}\right)^{\sim}_{-}\setminus\{N_0^-\}$.
\item[(b)] Let $N\in\left(\sqrt{N_0^2}\right)^{\sim}_{-}$, and let $U=\mat{u_1}{u_2}{u_3}{u_4}$ be an invertible matrix such that $UNU^{-1}=N^-_0$. Then $u_2$ solves \eqref{sylhom} while $u_3$ solves \eqref{sylhomadj}.
\end{itemize}
\end{lemma}
\begin{proof}
\begin{itemize}
\item[(a)] Let $x$ be a given solution to the respective equation $ax=xb$. Denote by $N_x:=\mat{a}{x}{0}{-b}$. Then there exists a unique $\bar x$ such that $a\bar x+\bar xb=x$,   given via \eqref{sylsolm+}. This implies that
$$a\bar x=\int_0^{+\infty}e^{-ta}axe^{-tb}dt=\int_0^{+\infty}e^{-ta}xbe^{-tb}dt=\bar x b.$$ On the other hand, Roth's removal rule implies that
$$N_x=\mat{1_{\A_1}}{-\bar x}{0}{1_{\A_2}}\mat{a}{0}{0}{-b}\mat{1_{\A_1}}{\bar x}{0}{1_{\A_2}},$$
and
$$\left(N_x\right)^2=\mat{a}{x}{0}{-b}\mat{a}{x}{0}{-b}=\mat{a^2}{ax-xb}{0}{b^2}=N_0^2.$$
Therefore $N_x\in\left(\sqrt{N_0^2}\right)^{\sim}_{-}$, and $\bar x$ solves the equation \eqref{sylhom} while $0_\K$ is a solution to \eqref{sylhomadj}. Specially, if $x\neq0$, then $\bar x\neq0$ and $N_x\neq N_0^-$. Similarly, by choosing 
$$\ _yN:=\mat{a}{0}{y}{-b}$$
the same method applies and gives the similarity matrices $\mat{1_{\A_1}}{0}{\pm\bar y}{1_{\A_2}}$, where $b\bar y+\bar ya=y$ and $\bar y$ is determined via \eqref{sylsolm+}, ergo, $b\bar{y}=\bar{y}a$.
\item[(b)] Conversely, let $N$ be a square root of $N_0^2$ which is similar to $N^-_0$. Then there exists an invertible matrix $U=\mat{u_1}{u_2}{u_3}{u_4}$, such that $N=U^{-1}N^-_0U$. This gives
$$N_0^2=N^2=U^{-1}(N^-_0)^2U\Leftrightarrow UN_0^2=N_0^2U.$$
The latter implies that 
$u_2b^2=a^2u_2$ and $b^2u_3=u_3a^2$. Let $au_2=u_2b+v$, for some $v\in\M$. Then $a^2u_2=au_2b+av$, and similarly, from $u_2b=au_2-v$ it follows that $u_2b^2=(au_2-v)b$. Equating the two gives $$a^2u_2-u_2b^2=au_2b+av-au_2b+vb=av+vb=0,$$ which implies that $v=0$, since $\sigma(a)\cup\sigma(b)\subset\HR$. Analogously, it holds that $u_3a=bu_3$.
\end{itemize}
\end{proof}

\begin{theorem} \label{0,0} Let $a$ and $b$ be such that $\sigma(a)\cap\sigma(b)\neq\emptyset$, and $\sigma(a)\cup\sigma(b)\subset\HR$. The following statements are equivalent:
\begin{itemize}
\item[(a)] There exists a pair $(x,y)\in \M\times\K\setminus\{(0_\M,0_\K)\}$ which solves the system \eqref{couplhom}.
\item[(b)] There exists a nonprimary square root for $N_0^2$ which is similar to $N_0^-$.
\end{itemize}
\end{theorem}
\begin{proof} $(a)\Rightarrow(b):$ If $x$ is a nonzero solution to \eqref{sylhom}, then the matrix $N_x=\mat{a}{x}{0}{-b}$ from the previous Lemma \ref{homsylsqrtM} provides a nonprimary square root for $N_0^2$. If $x=0$ and $y$ is a nonzero solution to \eqref{sylhomadj}, then $\ _yN=\mat{a}{0}{y}{-b}$ from the same Lemma provides a nonprimary square root for $N_0^2$. 

$(b)\Rightarrow(a):$ Conversely, let $N$ be a nonprimary square root of $N_0^2$, which is similar to $N_0^-$. Then there exists an invertible matrix $U$ such that $N=UN_0^{-}U^{-1}$. If $u_{12}=0$ and $u_{21}=0$, then $N$ would be block-diagonal, i.e., it would be a primary square root of $N_0^2$, which is not the case. Therefore at least one entry out of $u_{12}$, $u_{21}$ is not zero. Then Lemma \ref{homsylsqrtM} $(b)$ applies and the proof is complete. 
\end{proof}

\subsection{The  subalgebra $\mathcal{M}_0$:\\ Solving the homogeneous equation}
As previously stated, one can choose $\K$ to be trivial, that is, $\K=0_{\K}$. Thus we introduce the following set in $\mathcal{M}$:
\begin{equation}\label{M0}\mathcal{M}_0=\left\{\mat{u_1}{u_2}{0}{u_4}: (u_1,u_2,u_4)\in\A_1\times\M\times\A_2,\  u_1a=au_1,\  bu_4=u_4b\right\}.\end{equation}
It is obvious that $1_{\mathcal{M}},\ N^{\pm}_0, \ N_0^2\in\mathcal{M}_0$.  It is not difficult to see that  $\mathcal{M}_0$ is a submonoind of $\mathcal{M}$, however, it is not abelian. 
\begin{lemma}\label{n0+} Let $a\in\A_1$ and $b\in\A_2$ be given not necessarily invertible elements, such that $\sigma(a)\cap\sigma(b)\neq\emptyset$. A matrix $\mat{u_1}{u_2}{0}{u_4}$ from $\mathcal{M}_0$ commutes with $N_0^+$, if and only if $u_{2}$ solves the homogeneous Sylvester equation \eqref{sylhom}.
\end{lemma} 
\begin{proof} This follows directly. \end{proof}
\noindent Combining the previous results, we obtain the following:
\begin{theorem}\label{main1} Let $a$ and $b$ be given elements from $\A_1$ and $\A_2$, respectively, such that $\sigma(a)\cap\sigma(b)\neq\emptyset$, and  $\sigma(a)\cup\sigma(b)\subset\HR$. The following statements are equivalent:
\begin{itemize}
\item[(a)] There exists a nonzero solution to \eqref{sylhom}.
\item[(b)] There exists a nonprimary square root for $N_0^2$ similar to $N_0^-$.
\item[(c)] There exists a matrix $U\in\mathcal{M}_0$ which is not block-diagonal, such that it commutes with $N_0^+$.
\end{itemize}
\end{theorem}
\begin{proof} Follows immediately from Theorem \ref{0,0} and Lemma \ref{n0+}.
\end{proof}

\begin{example} Let $h\in L^2(\RR)$ be a given nonzero function, and let $\widehat{h}\in L^2(\RR)$ be its Fourier transform. Denote by $C_h$ the convolution operator $C_h(f):=f*h$ in $L^2(\RR)$, and by $\mu_{\widehat{h}}$ the multiplication operator $\mu_{\widehat{h}}(f):=\widehat{h}f$ in $L^2(\RR)$. It is clear that the two bounded operators $\mu_{\widehat{h}}$ and $C_h$ are unitarily equivalent in $\B(L^2(\RR))$ by the virtue of the Fourier transform $\mathcal{F}\in\B(L^2(\RR))$:
$$\left(\mu_{\widehat{h}}\ \mathcal{F}\right)(f)=\mathcal{F}\left(C_h(f)\right)=\left(\mathcal{F} \ C_h\right)(f), \quad f\in L^2(\RR).$$ 
Therefore by Theorem \ref{main1} the operator matrix $N_0^2:=\mat{\mu^2_{\widehat{h}}}{0}{0}{C^2_h}$ has a nonprimary square root which is similar to $N_0^-:=\mat{\mu_{\widehat{h}}}{0}{0}{-C_h}$. Indeed, the matrix
$$N_{\mathcal{F}}:=\mat{\mu_{\widehat{h}}}{\mathcal{F}}{0}{-C_h}$$
is a nonprimary square root for $N_0^2$, similar to the matrix $N_0^-$ by the means of the similarity matrix
$$U_{\pm}=\mat{\id}{\pm\overline{\mathcal{F}}}{0}{\id},$$
where
$$\overline{\mathcal{F}}:=\int_0^{+\infty}e^{-t\mu_{\widehat{h}}}\mathcal{F}e^{-tC_h}dt,$$
and $\id$ is the identity operator on $L^2(\RR)$, that is, the unit in $\B(L^2(\RR))$. 

There can exist some other nonprimary square roots of $N_0^2$, depending on the choice for $h$: suppose that the function $h$ is chosen in the manner that
$$\int_{-\infty}^{+\infty} \sin (\xi t)h(\xi)d\xi=0, \quad \forall t\in\RR.$$
For instance, it suffices to let $h$ be an even function. Then we have
$$\mathcal{F}^{-1}(h)(t)=\frac{1}{\sqrt{2\pi}}\int_{-\infty}^{+\infty} e^{it\xi}h(\xi)d\xi=\frac{1}{\sqrt{2\pi}}\int_{-\infty}^{+\infty} e^{-it\xi}h(\xi)d\xi=\mathcal{F}(h)(t)=\widehat{h}(t),$$
which implies
$$\mathcal{F}^{-1}\left(h* f\right)=\mathcal{F}^{-1}(h)\cdot \mathcal{F}^{-1}(f)=\widehat{h}\cdot \mathcal{F}^{-1}(f),$$
i.e.
$$\left(\mathcal{F}^{-1}C_h - \mu_{\widehat{h}}\mathcal{F}^{-1}\right)(f)=0, \quad f\in L^2(\RR).$$
Since $h\neq0$ it follows that $\mathcal{F}^{-1}C_h\neq -\mu_{\widehat{h}}\mathcal{F}^{-1}$, and  
$$\begin{aligned}
&\mu_{\widehat{h}}^2\mathcal{F}^{-1}-\mathcal{F}^{-1}C_h^2=\mu_{\widehat{h}}^2\mathcal{F}^{-1}-\mathcal{F}^{-1}C_h^2-\mu_{\widehat{h}}\mathcal{F}^{-1}C_h+\mu_{\widehat{h}}\mathcal{F}^{-1}C_h=\\
=&\mu_{\widehat{h}}(\mu_{\widehat{h}}\mathcal{F}^{-1}-\mathcal{F}^{-1}C_h)+(\mu_{\widehat{h}}\mathcal{F}^{-1}-\mathcal{F}^{-1}C_h)C_h
=0.
\end{aligned}$$
Denote by 
$$N_{\left(\mu_{\widehat{h}}\mathcal{F}^{-1}+\mathcal{F}^{-1}C_h\right)}:=\mat{\mu_{\widehat{h}}}{\mu_{\widehat{h}}\mathcal{F}^{-1}+\mathcal{F}^{-1}C_h}{0}{-C_h}.$$
We have
$$\begin{aligned}
&\mat{\mu_{\widehat{h}}}{\mathcal{F}^{-1}C_h+\mu_{\widehat{h}}\mathcal{F}^{-1}}{0}{-C_h}\mat{\mu_{\widehat{h}}}{\mathcal{F}^{-1}C_h+\mu_{\widehat{h}}\mathcal{F}^{-1}}{0}{-C_h}=\\
=&\mat{\mu_{\widehat{h}}^2}{\mu_{\widehat{h}}^2\mathcal{F}^{-1}-\mathcal{F}^{-1}C_h^2}{0}{C_h^2}=\mat{\mu_{\widehat{h}}^2}{0}{0}{C_h^2},
\end{aligned}$$
ergo  $\mat{\mu_{\widehat{h}}}{\mathcal{F}^{-1}C_h+\mu_{\widehat{h}}\mathcal{F}^{-1}}{0}{-C_h}$ is a nonprimary square root of $N_0^2$. It also holds that
$$\mat{\id}{-\mathcal{F}^{-1}}{0}{\id}\mat{\mu_{\widehat{h}}}{0}{0}{-C_h}\mat{\id}{\mathcal{F}^{-1}}{0}{\id}=\mat{\mu_{\widehat{h}}}{\mu_{\widehat{h}}\mathcal{F}^{-1}+\mathcal{F}^{-1}C_h}{0}{-C_h},$$
thus concluding that $N_{\left(\mu_{\widehat{h}}\mathcal{F}^{-1}+\mathcal{F}^{-1}C_h\right)}$ and $N_0^-$ are indeed similar.
Also note that
$$\mat{\mu_{\widehat{h}}}{0}{0}{C_h}\mat{\id}{\mathcal{F}}{0}{\id}-\mat{\id}{\mathcal{F}}{0}{\id}\mat{\mu_{\widehat{h}}}{0}{0}{C_h}=\mat{0}{0}{0}{0},$$
so $\mat{\id}{\mathcal{F}}{0}{\id}$ commutes with $N_0^+=\mat{\mu_{\widehat{h}}}{0}{0}{C_h}$.
\  \hfill $\clubsuit$
\end{example}

\subsection{The group algebra $\mathcal{M}_0^{-1}$:\\Solving the inhomogeneous equation }
In analogy with the previous subsection, below we obtain the ``if and only if'' conditions for the consistency of \eqref{sylm-} in terms of the matrix algebra $\mathcal{M}_0^{-1}$. 

For given elements $a\in\A_1$ and $b\in\A_2$, let $N_0^\pm$ and $N_0^2$ be respectively provided via \eqref{N0} and \eqref{M}. Note that for any $N$ which is similar to $N_0^-$, there exists an invertible $U\in\mathcal{M}^{-1}$, such that $U^{-1}N_0^- U=N$. We introduce the following structure:
\begin{equation}\label{M0-1}\mathcal{M}^{-1}_0=\left\{\mat{u_1}{u_2}{0}{u_4}\in\mathcal{M}_0: \ u_1, u_4-\textrm{invertible}\right\}.\end{equation}
It is straightforward to see that $1_{\mathcal{M}}\in\mathcal{M}^{-1}_0$, while $N^\pm_0, N_0^2\in\mathcal{M}_0^{-1}$ if
 and only if $a$ and $b$ are invertible in the respective algebras. Moreover, notice that for any $U\in\mathcal{M}^{-1}_0$, where $U=\mat{u_1}{u_2}{0}{u_4}$, the matrix $U$ itself is invertible in $\mathcal{M}_0$, and its inverse is given as
$$U^{-1}=\mat{u_1^{-1}}{-u_1^{-1}u_2 u_4^{-1}}{0}{u_4^{-1}},$$
therefore $U^{-1}\in\mathcal{M}^{-1}_0$ as well. This proves the following:
\begin{proposition} With respect to the previous notation, the ordered triplet $\left(\mathcal{M}^{-1}_0,\circ,1_{\mathcal{M}}\right)$ is a group with the unit $1_{\mathcal{M}}$.
\end{proposition}
\noindent As our results below demonstrate, the algebra $\mathcal{M}_0^{-1}$ is more than enough to inspect the consistency of \eqref{sylm-}. Recall the sets from $\Omega$, $\Omega_0$, and $\Omega_c$ from \eqref{Omega}, \eqref{Omega0}, and \eqref{OmegaC}, respectively:
$$\Omega=\{(z,w): az=wb\},\ \Omega_0:=\{(z,z): az=zb\},\ \Omega_c=\{(z,w)\in\Omega: w-z=c\}.$$

\begin{lemma}\label{MN1N2C}  Let $a$ and $b$ be invertible in the respective algebras, such that $\sigma(a)\cap\sigma(b)\neq\emptyset$, and $\sigma(a)\cup\sigma(b)\subset\HR$. The following statements are true:
\begin{itemize}
\item[(a)] For every ordered pair $(z,w)\in\Omega$ there exist matrices $U$, $V\in\mathcal{M}^{-1}_0$ given as
$$U=\mat{u_1}{u_2}{0}{u_4},\quad V=\mat{v_1}{v_2}{0}{v_4},$$
which  satisfy
$$N_0^2=\left(U^{-1}\circ N^-_0\circ U\right)\circ\left(V\circ N^-_0\circ V^{-1}\right),$$
while the expressions $-\left(au_1v_2+u_1v_2b\right)$ and $\left(au_2v_4+u_2v_4b\right)$ are equal to $z$ and $w$ respectively, i.e., 
$$-(au_1v_2+u_1v_2b)=z,\  au_2v_4+u_2v_4b=w.$$

\item[(b)]Let $N_0^2=N_1\circ N_2$, where
$$N_1=U^{-1}\circ N^-_0\circ U,$$
$$N_2=V\circ N^-_0\circ V^{-1},$$ for some matrices $U=\mat{u_1}{u_2}{0}{u_4}\in \mathcal{M}^{-1}_0$ and $V=\mat{v_1}{v_2}{0}{v_4}\in \mathcal{M}^{-1}_0$. 
Then 
$$\left(\pm\left(au_1v_2+u_1v_2b\right), \mp\left( u_2v_4b+au_2v_4\right)\right)\in\Omega.$$
\end{itemize}
\end{lemma}
\begin{proof}
\begin{itemize}
\item[(a)] For easier orientation, let $r_1:=w$ and $r_2:=z$. For $i\in\{1,2\}$, let
$$N_{r_i}=\left[\begin{array}{cr}a & r_i\\ 0 & -b\end{array}\right].$$
Then 
$$N_{r_1}\circ N_{r_2}=\left[\begin{array}{cr}a & r_1\\ 0 & -b\end{array}\right]\circ \left[\begin{array}{cr}a & r_2\\ 0 & -b\end{array}\right]=\left[\begin{array}{cc}a^2 & ar_2-r_1b\\ 0 & b^2\end{array}\right]=N_0^2.$$
Since $\sigma(a)\cup\sigma(b)\subset\HR$, there exists a unique $p_i$ such that
$$ap_i+p_ib=r_i,$$
given by \eqref{sylsolm+} as
$$p_i=\int_0^\infty e^{-ta}r_ie^{-tb}dt.$$
Notice that if $ar_i\neq r_ib$, then $ap_i\neq p_ib$. However, by Roth's removal rule we have
$$N_{r_i}=\mat{1_{\A_1}}{-p_i}{0}{1_{\A_2}}\circ\mat{a}{0}{0}{-b}\circ\mat{1_{\A_1}}{p_i}{0}{1_{\A_2}},\quad i=1,2.$$
Since $N_0^2=N_{r_1}\circ N_{r_2}$, in that order, we choose the matrices $U$ and $V$ to be
$$U:=\mat{1_{\A_1}}{p_1}{0}{1_{\A_2}},\quad V:=\mat{1_{\A_1}}{-p_2}{0}{1_{\A_2}}\in\mathcal{M}^{-1}_0.$$
Then direct calculations give
$$-au_1v_2-u_1v_2b=ap_2+p_2b=r_2$$
and
$$au_2v_4+u_2v_4b=ap_1+p_1b=r_1.$$
\item[(b)] Conversely, let $N_1$ and $N_2$ be such that $N_0^2=N_1\circ N_2$, where $N_1=U^{-1}\circ N^-_0\circ U$ while $N_2=V\circ N^-_0\circ V^{-1}$, for 
$$U:=\mat{u_1}{u_2}{0}{u_4},\ V:=\mat{v_1}{v_2}{0}{v_4}\in\mathcal{M}_0^{-1}.$$ 
Then $N_0^2=N_1\circ N_2=U^{-1}\circ N^-_0\circ U\circ V\circ N^-_0\circ V^{-1}$ gives
$$\begin{aligned}
&\mat{u_1}{u_2}{0}{u_4}\circ \mat{a^2}{0}{0}{b^2}\circ \mat{v_1}{v_2}{0}{v_4}=\\
&\mat{a}{0}{0}{-b}\circ \mat{u_1v_1}{u_1v_2+u_2v_4}{0}{u_4v_4}\circ \mat{a}{0}{0}{-b}\Leftrightarrow\\
&\mat{u_1a^2v_1}{u_1a^2v_2+u_2b^2v_4}{0}{u_4b^2v_4}=\mat{au_1v_1a}{-au_1v_2b-au_2v_4b}{0}{bu_4v_4b}.
\end{aligned}$$
The latter implies that 
$$\begin{aligned}
&u_1a^2v_2+u_2b^2v_4=-au_1v_2b-au_2v_4b\Leftrightarrow\\
&a(au_1v_2+u_1v_2b)+(u_2v_4b+au_2v_4)b=0\end{aligned}$$
i.e., 
$$\left(\pm\left(au_1v_2+u_1v_2b\right), \mp\left( u_2v_4b+au_2v_4\right)\right)\in\Omega.$$
\end{itemize}
\end{proof}
\noindent Combining the previous results, we obtain the following theorem:

\begin{theorem}[Solvability and particular solution-matrix version]\label{part} Let $a$ and $b$ be invertible in $\A_1$ and $\A_2$, respectively, such that $\sigma(a)\cap\sigma(b)\neq\emptyset$, and with the property that $\sigma(a)\cup\sigma(b)\subset\HR$. For an arbitrary $c\in\M$, let $\bar{c}$ be the unique solution to the equation
\begin{equation}
\label{sylC}a\bar{c}+\bar{c}b=c.\end{equation} 
Then:
\begin{enumerate}
\item The following statements are equivalent:
\begin{itemize}
\item[(a)] The equation \eqref{sylm-} is solvable.
\item[(b)] There exist $U=\mat{u_1}{u_2}{0}{u_4}$ and $V=\mat{v_1}{v_2}{0}{v_4}$ in $\mathcal{M}^{-1}_0$, such that
\begin{equation}\label{MUN0V}N_0^2=\left(U^{-1}\circ N^-_0\circ U\right)\circ\left(V\circ N^-_0\circ V^{-1}\right),\end{equation}
and 
\begin{equation}
\label{uivjbarc}u_1v_2+u_2v_4=\bar{c}.
\end{equation}

\item[(c)] There exist $U'=\mat{a}{u}{0}{b'}$ and $V'=\mat{a'}{v}{0}{b}$ in $\mathcal{M}^{-1}_0$, with $a'\in (a)'$ and $b'\in (b)'$ arbitrary, such that
\begin{equation}\label{U'V'}
U'N_0^2V'=N^-_0 U'V'N^-_0,\end{equation}
and
\begin{equation}
\label{gensyl0}av+ub=\bar{c}.
\end{equation} 
\end{itemize}
\item If the equation \eqref{sylm-} is solvable, then with respect to the claim $(b)$, the expressions
\begin{equation}\label{solu12v12}
x_{u_1v_2}:=-(au_1v_2b^{-1}+u_1v_2);\quad x_{u_2v_4}:=a^{-1}u_2v_4b+u_2v_4
\end{equation}
define the same particular solution to \eqref{sylm-}. Alternatively, with respect to the claim $(c)$,  the expressions
\begin{equation}\label{soluvab}
x_v:=-(a^2vb^{-1}+av);\quad x_u:=a^{-1}ub^2+ub
\end{equation}
also define one and the same particular solution to \eqref{sylm-}.
\end{enumerate}
\end{theorem}
\begin{proof}\begin{enumerate}
\item $(a)\Leftrightarrow(b):$ By Lemma \ref{C1C2Y}, the equation \eqref{sylm-} is solvable if and only if there exist elements $c_1$ and $c_2$ in $\M$, such that $c=c_1-c_2$ while $(c_2,c_1)\in\Omega$. On the other hand, by Lemma \ref{MN1N2C}, $(c_2,c_1)\in\Omega$ if and only if there exist matrices $U$ and $V$ in $\mathcal{M}^{-1}_0$, such that \eqref{MUN0V} holds, in the manner that  $c_1=u_2v_4b+au_2v_4$, while $c_2=-au_1v_2-u_1v_2b$. Finally, the condition $c_1-c_2=c$ holds if and only if
$$\begin{aligned}
&u_2v_4b+au_2v_4+au_1v_2+u_1v_2b=c\Leftrightarrow\\
&a(u_1v_2+u_2v_4)+(u_1v_2+u_2v_4)b=c.\end{aligned}$$

\noindent$(b)\Rightarrow (c):$ Assume that $U$ and $V$ are provided such that \eqref{MUN0V} holds and that $u_1v_2+u_2v_4=\bar{c}$. Then by choosing $v:=a^{-1}u_1v_2$ and $u:=u_2v_4b^{-1}$, it follows that
$$av+ub=u_1v_2+u_2v_4=\bar{c}.$$
Similarly, from \eqref{MUN0V} we have
$$\mat{u_1v_1a^2}{a^2u_1v_2+u_2v_4b^2}{0}{u_4v_4b^2}=\mat{u_1v_1a^2}{-au_1v_2b-au_2v_4b}{0}{u_4v_4b^2},$$
therefore
\begin{eqnarray}\label{consistency}\begin{aligned}
&a^2u_1v_2+u_2v_4b^2+au_1v_2b+au_2v_4b=0\Leftrightarrow\\
&a^3v+ub^3+a^2vb+aub^2=0\Leftrightarrow\\
&a^3v+ub^3=-a^2vb-aub^2.\end{aligned}
\end{eqnarray}
The latter shows that the equality
\begin{equation}\label{matmain}\mat{a^3a'}{a^3v+ub^3}{0}{b'b^3}=\mat{a^3a'}{-a^2vb-aub^2}{0}{b'b^3}\end{equation}
is true, for any $a'\in(a)'$, and any $b'\in(b)'$. Therefore, the equality \eqref{U'V'} is true as well.\\

$(c)\Rightarrow(b)$: obvious.
\item Finally, if the equation \eqref{sylm-} is solvable then both $(b)$ and $(c)$ are true. With respect to the claim $(b)$, Lemma \ref{C1C2Y} states that a particular solution can be chosen to be given as \eqref{solu12v12}, due to \eqref{solparinv}. On the other hand, with respect to the claim $(c)$, by adjusting the formulas \eqref{solu12v12} and \eqref{solparinv} to the setting where $u_1=a$, $v_4=b$, $u_2=u$, and $v_2=v$, we obtain the formula \eqref{soluvab} for a particular solution. 
\end{enumerate}
\end{proof}

\noindent Even though Theorem \ref{part} is in practice inadequate, it still provides enough information about a very important triplet for which the Sylvester equation is not solvable. While the following result is a well-known claim, we proceed to prove it via Theorem \ref{part} as well.

\begin{corollary} Let $\A$ be a complex unital Banach algebra. Then the equation
\begin{equation}
\label{axxa1}
ax-xa=1_\A
\end{equation}
is not solvable in $\A$. 
\end{corollary}
\begin{proof}
Assume conversely: let $x\in\A$ be a solution to \eqref{axxa1}. Then, for any $\lambda \in\CC$, the solution $x$ also solves the equation $(a-\lambda 1_\A)x-x(a-\lambda 1_{\A})=1_\A$, thus, without loss of generality, we assume that $\sigma(a)\subset\HR$ as before. For any solution $x$ to $ax-xa=1_\A$, it follows that (see also \cite{HC}) $a^2x=a+axa$, while $-xa^2=a-axa$, therefore
\begin{equation}
\label{chara}
a^2x-xa^2=a+a=2a.
\end{equation}
On the other hand, due to the claim $(b)$ of Theorem \ref{part}, there exist matrices 
$$U=\mat{u_1}{u_2}{0}{u_4}, \quad V=\mat{v_1}{v_2}{0}{v_4}\in\mathcal{M}_0^{-1}$$
such that
$$U\mat{a^2}{0}{0}{a^2}V=\mat{a}{0}{0}{-a}UV\mat{a}{0}{0}{-a}.$$
Computing the above entries gives
$$U\mat{a^2}{0}{0}{a^2}V=\mat{u_1 a^2v_1}{u_1a^2v_2+u_2a^2v_4}{0}{u_4a^2v_4}$$
and
$$\mat{a}{0}{0}{-a}UV\mat{a}{0}{0}{-a}=\mat{au_1v_1a}{-a(u_1v_2+u_2v_4)a}{0}{au_4v_4a}.$$
The two are equal if and only if
\begin{equation}\label{char}a^2u_1v_2+u_2v_4a^2+a(u_1v_2+u_2v_4)a=0.\end{equation}
Since the expression $u_1v_2+u_2v_4$ solves the regular equation $a\bar{c}+\bar{c}a=1_\A$, it is calculated via \eqref{sylsolm+} as
$$u_1v_2+u_2v_4=\int_0^\infty e^{-ta}1_\A e^{-ta}dt=\int_0^\infty e^{-2ta}dt=\frac{1}{2}a^{-1}.$$
Substituting this into \eqref{char} gives
\begin{equation}
\label{charb}
a^2u_1v_2+u_2v_4a^2=-\frac{1}{2}a.\end{equation}
Finally, recall that the solution $x$ can be chosen to be given as
\begin{equation}\label{charc}x=u_2v_4+au_2v_4a^{-1}=-u_1v_2-a^{-1}u_1v_2a.\end{equation}
Therefore, 
\begin{equation}
\label{chard}
a^2x=-a^2u_1v_2-au_1v_2a,\quad xa^2=u_2v_4a^2+au_2v_4a.\end{equation}
Combining \eqref{chara}, \eqref{char}, and \eqref{chard} gives
$$-2a=a^2u_1v_2+au_1v_2a+u_2v_4a^2+au_2v_4a=0$$
which is impossible.
\end{proof}

Given that  in Theorem \ref{part} the entries $a'\in(a)'$ and $b'\in(b)'$ can be chosen arbitrarily, the following corollary immediately follows from \eqref{consistency}:

\begin{corollary}\label{gensystemsol} With respect to the notation and assumptions from Theorem \ref{part}, the equation \eqref{sylm-} is solvable if and only if there exist $u$ and $v\in\M$ such that the system is consistent:
\begin{equation}\label{systemforuandv}
\begin{cases}av+ub=\overline{c}\\
a^3v+a^2vb+ub^3+aub^2=0.
\end{cases}
\end{equation}
In that case,   the expressions
\eqref{soluvab} define a (one and the same) particular solution to \eqref{sylm-}.
\end{corollary}
\begin{proof} By Theorem \ref{part}, the equation \eqref{sylm-} is solvable if and only if there exist $u$ and $v\in\M$ such that \eqref{gensyl0} holds, and simultaneously the matrix equality \eqref{matmain} holds. Comparing the positions $(1,2)$ of the matrices in \eqref{matmain}, we see conclude that \eqref{sylm-} is solvable if and only if there exist $u$ and $v\in\M$ such that \eqref{systemforuandv} is consistent. \end{proof}
Notice that when $u=v$, the second equation from the system \eqref{systemforuandv} reduces to the regular generalized Sylvester equation
\begin{equation}
\label{RajUch} a^3u+a^2ub+aub^2+ub^3=0, 
\
\end{equation}
which has the unique solution $u=0$ when $\sigma(a)\cup\sigma(b)\in\Lambda_{\pi/3} $ (see \cite{RBMU} for details).
However, the expression \eqref{soluvab} prevents the possibility that $u=v$, whenever $c\neq0$.

\section{Finding the parameters $u$ and $v$}
By Corollary \ref{gensystemsol}, solvability of \eqref{sylm-} is equivalent to the consistency of \eqref{systemforuandv}, therefore in this section we focus on finding convenient forms for the missing parameters $u$ and $v$ which solve the latter. In order to reduce repetition, throughout this subsection it is assumes that $a$ and $b$ are invertible in $\A_1$ and $\A_2$, respectively, such that $\sigma(a)\cap\sigma(b)\neq\emptyset$ (though this is not a necessary condtion, but merely an emphasis on the singularity of the problem), and with the property that $\sigma(a)\cup\sigma(b)\subset\HR$. Moreover, we assume that $c\in\M$ is provided (and fixed), and by $\bar{c}$ we denote the unique solution to the equation 
\begin{equation}\label{barc}a\bar{c}+\bar{c}b=c.\end{equation}
For short, we denote by
\begin{equation}\label{r}
r:=a^{-1}\bar{c}b+a\bar{c}b^{-1}.
\end{equation}
Respectively, for arbitrary $u$ and $v$ in $\M$, we denote the following relations:
\begin{equation}\label{gensyl}
av+ub=\bar{c}.
\end{equation}
\begin{equation}\label{gensylvu}
au+vb=\bar{c}+a^{-1}\bar{c}b+a\bar{c}b^{-1}.
\end{equation}
 \begin{equation}\label{u+v}
u+v=a^{-1}cb^{-1}.
\end{equation}
 \begin{equation}\label{RajUchiuv}
a^3v+a^2vb+ub^3+aub^2=0.
\end{equation}

\begin{lemma}\label{2of4} Let $a$, $b$, $c$, and $\bar{c}$ be provided as above. Let $u$ and $v\in\M$ be arbitrary. Then the chain of equivalences is true for $u$ and $v$:
\begin{eqnarray}
\label{equivalence}
\begin{aligned}
&\left(\eqref{gensyl}\wedge\eqref{gensylvu}\right)\Leftrightarrow\left(\eqref{u+v}\wedge\eqref{gensylvu}\right)\Leftrightarrow\left(\eqref{gensyl}\wedge\eqref{u+v}\right)\\
&\Leftrightarrow\left(\eqref{RajUchiuv}\wedge\eqref{u+v}\right)\Leftrightarrow\left(\eqref{gensyl}\wedge\eqref{RajUchiuv}\right).
\end{aligned}
\end{eqnarray} 
In other words, if $u$ and $v$ solve any two equations enlisted in \eqref{equivalence}, then they solve all of them.
\end{lemma}
\begin{remark} At this point, the only implication missing is that if some $u$ and $v$ solve the equations \eqref{gensylvu} and \eqref{RajUchiuv}, then the pair $(u,v)$ solves all of them.This will be proven to be true shortly, though under a slightly stronger condition, which is always feasible.
\end{remark}
\begin{proof}
$$\left(\eqref{gensyl}\wedge\eqref{gensylvu}\right) \Longleftrightarrow  \left(\eqref{u+v}\wedge\eqref{gensylvu}\right)
\Longleftrightarrow\left(\eqref{u+v}\wedge\eqref{gensyl}\right):$$
From \eqref{gensyl} and \eqref{gensylvu} it follows that
$$\begin{aligned}
&a(u+v) + (u+v)b
= 2\bar{c}
+ \int_0^\infty e^{-at}(a^{-1} cb + acb^{-1})e^{-bt}dt=\\
=& \int_0^\infty e^{-at}(c + a^{-1}cb + c + acb^{-1})e^{-bt}dt=\\
=&\int_0^\infty e^{-at}\left(a(a^{-1}c) + a^{-1}cb\right)e^{-bt}dt+\int_0^\infty e^{-at}\left(a(cb^{-1}) + (cb^{-1})b\right)e^{-bt}dt=\\
=& a^{-1}c + cb^{-1}.
\end{aligned}$$
Since the transform $(u+v)\mapsto a(u+v)+(u+v)b$ is regular, it follows that $u+v=a^{-1}cb^{-1}$ is the only solution to the above equation, that is, \eqref{u+v} is true. 

Similarly, if \eqref{u+v} holds, then by the calculations above we have:
$$a(u+v)+(u+v)b=a^{-1}c+cb^{-1}=2\bar{c}
+ \int_0^\infty e^{-at}(a^{-1} cb + acb^{-1})e^{-bt}dt.$$
Ergo \eqref{gensyl} is true if and only if \eqref{gensylvu} is also true. 
$$\left(\eqref{gensyl}\wedge\eqref{RajUchiuv}\right)\Longleftrightarrow\left(\eqref{RajUchiuv}\wedge\eqref{u+v}\right)\Longleftrightarrow\left(\eqref{gensyl}\wedge\eqref{u+v}\right):$$
Assume that \eqref{gensyl} holds. Then  $$\begin{aligned}
&a^3v+ub^3+a^2vb+aub^2=\\
&a(a^2v+ub^2)+(a^2v+ub^2)b=\\
&a\left(a\bar{c}+\bar{c}b-a(u+v)b\right)+\left(a\bar{c}+\bar{c}b-a(u+v)b\right)b=\\
&a\left(c-a(u+v)b\right)+\left(c-a(u+v)b\right)b.\end{aligned}$$ 
Therefore, \eqref{RajUchiuv} holds if and only if \eqref{u+v} holds.

Conversely, assume that \eqref{RajUchiuv} and \eqref{u+v} hold. Then 
$$\begin{aligned}
&0=a^3v+a^2vb+ub^3+aub^2=\\
&a(a^2v+ub^2)+(a^2v+ub^2)b\Leftrightarrow\\
&a^2v+ub^2=0\Leftrightarrow\\
&a^2v+ub^2=c-a(u+v)b\Leftrightarrow\\
&a^2v+ub^2+aub+avb=c\Leftrightarrow\\
&a(av+ub)+(av+ub)b=c\Leftrightarrow\\
&av+ub=\overline{c}.
\end{aligned}$$ 

\end{proof}
Since $a$ and $b$ are invertible, note that each equation from \eqref{gensyl}--\eqref{RajUchiuv} is solvable, each with infinitely many pairwise unique solutions: if one of the entries $u,v$ is chosen arbitrarily, then it uniquely determines the other one within that equation. Thus, the goal is to find the pairs $(u,v)$ which satisfy any two conditions enlisted in \eqref{equivalence}. 
 \begin{lemma}\label{uvsyl}
Let $a$, $b$, $c$, and $\bar{c}$ be appropriately given elements, and let $r$ be provided as in \eqref{r}. For arbtirary two elements $u,v\in\M$ the following statements are equivalent: 
\begin{itemize}
\item[(a)] The ordered pair $(u,v)$ satisfies any two conditions enlisted in \eqref{equivalence}.
\item[(b)] The element $u$ satisfies
\begin{equation}
\label{auub}
au-ub=a\bar{c}b^{-1}\end{equation}
\item[(c)] The element $v$ satisfies:
\begin{equation}\label{avvb}
av-vb=-a^{-1}\bar{c}b.
\end{equation}
\end{itemize}
\end{lemma}

\begin{proof} We show the equivalence $(a)\Longleftrightarrow(b)$. The equivalence $(a)\Longleftrightarrow(c)$ is analogous.

$(b)\Rightarrow(a):$ Assume that \eqref{auub} holds. Then
$$\begin{aligned}
&a^2u-ub^2=a(au-ub)+(au-ub)b=\\
=&a\left(\int_0^{+\infty} e^{-at}acb^{-1}e^{-bt}dt\right)+\left(\int_0^{+\infty} e^{-at}acb^{-1}e^{-bt}dt\right)b=\\
=&\int_0^{+\infty}e^{-at}(a^2cb^{-1}+ac)e^{-bt}dt,
\end{aligned}$$
therefore
\begin{eqnarray}
\label{b2}
\begin{aligned}
&aub^{-1}-a^{-1}ub=a^{-1}\left(a^2u-ub^2\right)b^{-1}=\int_0^{+\infty}e^{-at}(acb^{-2}+cb ^{-1})e^{-bt}dt=\\
=&\int_0^{+\infty}e^{-at}\left(a\left(cb^{-2}\right)+\left(cb ^{-2}\right)b\right)e^{-bt}dt=cb^{-2}.\end{aligned}\end{eqnarray}
Since $r$ is the unique solution to $ar+rb=a^{-1}cb+acb^{-1}$, it is easy to see that
$$cb^{-1}+a^{-1}\bar{c}b-\bar{c}=r.$$
However, this implies that 
$$cb^{-2}=rb^{-1}+\bar{c}b^{-1}+a^{-1}\bar{c}.$$
Substituting the latter into \eqref{b2} gives
\begin{equation}
\label{invdiff}
aub^{-1}-a^{-1}ub=\bar{c}b^{-1}+rb^{-1}-a^{-1}\bar{c},\end{equation}
or, equivalently,
$$a^{-1}\bar{c}-a^{-1}ub=\bar{c}b^{-1}+rb^{-1}-aub^{-1}.$$
Set 
$$v:=a^{-1}\bar{c}-a^{-1}ub\equiv \bar{c}b^{-1}+rb^{-1}-aub^{-1}.$$ 
Then $av+ub=\bar{c}$ and $au+vb=\bar{c}+r$. By Lemma \ref{2of4} the pair $(u,v)$ satisfies all the conditions in \eqref{equivalence}.\\

$(a)\Rightarrow(b):$ Let $u$ and $v$ in $\M$ be provided such that \eqref{gensyl} and \eqref{RajUchiuv} hold. Denote by   $d:=au - ub$. Then $au = ub + d$, $aub^2 = ub^3 + db^2$, and $ub^3 = aub^2 - db^2$.
On the other hand, we have $av = \bar{c} - ub$, $
a^2 vb = a\bar{c}b - aub^2$, and $a^3 v = a^2 \bar{c} - a^2 ub$. Therefore
$$\begin{aligned}
0=&a^3 v + a^2 vb + aub^2 + ub^3
=\\
=&a^3 v + a^2 vb + 2aub^2 - db^2
=\\
=&a^2 \bar{c} - a^2 ub + a\bar{c}b - aub^2 + 2aub^2 - db^2=\\
=&a(a\bar{c} + \bar{c}b) + aub^2 - a^2 ub - db^2=\\
=&ac - \left(a(au - ub)b + db^2\right)=ac -\left(adb + db^2\right)=\\
=&ac - (ad + db)b.
\end{aligned}
$$
This is equivalent to
\[
acb^{-1} = ad + db
\Longleftrightarrow
au - ub = \int_0^\infty e^{-at}\, a c b^{-1}\, e^{-bt}\, dt.
\] 
\end{proof}
We  collect the previous analysis within the following theorem: 
\begin{theorem}[System for $u$ and $v$]\label{4of2} Let $a$, $b$, $c$, $\bar{c}$, and $r$ be appropriately given elements, such that $\sigma(a)\cup\sigma(b)\subset\Lambda_{\pi/3}$. If there exist elements $u$ and $v$ in $\M$ which solve any two equations of \eqref{gensyl}--\eqref{RajUchiuv}, then the pair $(u,v)$ solves all of them.
\end{theorem}

\begin{proof} By Lemma \ref{2of4} it suffices to show that \eqref{gensylvu} and \eqref{RajUchiuv} imply the rest. Indeed, from \eqref{RajUchiuv} it follows that
$$a(a^2v+ub^2)+(a^2v+ub^2)b=0\Leftrightarrow v=-a^{-2}ub^2.$$
On the other hand, from \eqref{gensylvu} we have 
$$a^2u+avb+aub+vb^2=a\bar{c}+\bar{c}b+ar+rb=c+a^{-1}cb+acb^{-1}.$$
Substituting $v=-a^{-2}ub^{2}$ into the latter gives
$$\begin{aligned}&a^{-1}(a^3u-ub^3)+a^{-2}(a^3u-ub^3)b=ca^{-1}cb+acb^{-1}\Leftrightarrow \\
&a(a^3-ub^3)+(a^3u-ub^3)b=a^2c+acb+a^3cb^{-1}\Leftrightarrow\\
&a^3u-ub^3=a^2\bar{c}+a\bar{c}b+a^3\bar{c}b^{-1}.
\end{aligned}$$
Now by invoking the mutliplication operators $Au=au$ and $Bu=ub$, it follows that (since they commute)
$$\begin{aligned}&a^3u-ub^3=(A^3-B^3)(u)=(A-B)(A^2+AB+B^2)(u)=\\
&(A^2+AB+B^2)(A-B)(u)=(A^2+AB+B^2)(au-ub).
\end{aligned}$$
The transform 
$$\xi\mapsto (A^2+AB+B^2)(\xi)$$
is regular under the premise that $\sigma(a)\cup\sigma(b)\subset\Lambda_{\pi/3}$ (see \cite{RBMU} for details), which implies the existence of a unique $\xi\in\M$ which solves
$$(A^2+AB+B^2)(\xi)=a^2\bar{c}+a\bar{c}b+a^3\bar{c}b^{-1}.$$
By a direct verification we see that $\xi=a\bar{c}b^{-1}$ is a good (and hence the only) candidate for the solution, ergo
$$a^3u-ub^3=a^2\bar{c}+a\bar{c}b+a^3\bar{c}b^{-1}\Leftrightarrow au-ub=a\bar{c}b^{-1}.$$
By Lemma \ref{uvsyl} it follows that all conditions enlisted in \eqref{equivalence} are true, hence, the pair $(u,v)$ indeed solves all equations \eqref{gensyl}--\eqref{RajUchiuv}.
\end{proof}

\subsection{The parameter $v-u$: a quadratic matrix equation}

Within this subsection we obtain a condition for the difference $q:=v-u$, which will be consistent with the equations \eqref{gensyl}--\eqref{RajUchiuv}. Though one might be tempted to seek an explicit form for the missing parameter $q$, it should be noted that the system 
$$\begin{cases} u+v=a^{-1}cb^{-1};\\v-u=q\end{cases}$$
is always solvable for $(u,v)$ and with a pairwise unique solution, which would imply that \eqref{sylm-} is always solvable with a unique solution (possibly provided by \eqref{soluvab}). Since this is not the case, the premise that $q$ can be explicitly written down needs to be discarded. Instead, we do the next best thign as demonstrated below. \\

By premise, the elements $a$ and $b$ are invertible in the respective algebras, with their spectrums contained in $\HR$. Thus the sets $\sqrt{a}$ and $\sqrt{b}$ are non-empy, and every $\xi\in\sqrt{a}$ and every $\eta\in\sqrt{b}$ is invertible as well, where again
$$\sqrt{a}=\{\xi\in\A_1: \xi^2=a\},\quad \sqrt{b}=\{\eta\in\A_2: \eta^2=b\}.$$
Then the element $-b$ also has square roots, where the setwise equality holds:
$$\sqrt{-b}=i\sqrt{b}.$$
Denote by 
$$N_1:=\mat{a}{-\bar{c}}{0}{-b},\quad N_2:=\mat{a}{-\bar{c}-r}{0}{-b}.$$
Since $\sigma(a)\cap\sigma(-b)=\emptyset$, there exist unique $e_1$ and $e_2\in\M$ such that
$$ae_1+e_1b=-\bar{c},\quad ae_2+e_2b=-\bar{c}-r.$$
Then by Roth's removal theorem:
$$N_i=\mat{1_{\A_1}}{-e_i}{0}{1_{\A_2}}\mat{a}{0}{0}{-b}\mat{1_{\A_1}}{e_i}{0}{1_{\A_2}},\quad i=1,2,$$
which implies that the matrices $N_i$ have square roots in $\mathcal{M}^{-1}_0$: indeed, one can take
$$N_i^{1/2}(k_1,k_2):=\mat{1_{\A_1}}{-e_i}{0}{1_{\A_2}}\mat{(-1)^{k_1}\sqrt{a}}{0}{0}{(-1)^{k_2}i\sqrt{b}}\mat{1_{\A_1}}{e_i}{0}{1_{\A_2}},$$
where $k_1$ and $k_2$ are integers independent from one another. This demonstrates that the sets 
$$\sqrt{N_i}:=\{ \Sigma_i\in\mathcal{M}^{-1}_0: \Sigma_i\circ\Sigma_i=N_i\},\quad i=1,2,$$
are non-empty and comprise out of invertible elements in $\mathcal{M}_0^{-1}$ as well. Ergo, the $\mathcal{M}_0^{-1}-$quadratic equation
\begin{equation}\label{YNY}YN_1Y=N_2\end{equation}
is solvable in $\mathcal{M}^{-1}_0$: by employing invertibility, \eqref{YNY} is mathematically equivalent to
\begin{equation}\label{N12}
N_1^{1/2}YN_1^{1/2}N_1^{1/2}YN_1^{1/2}=N_1^{1/2}N_2N_1^{1/2},\quad N_1^{1/2}\in\sqrt{N_1}-\textrm{arbitrary},\end{equation}
and the latter reads
\begin{equation}\label{Y2}
\left(N_1^{1/2}YN_1^{1/2}\right)^2=N_1^{1/2}N_2N_1^{-1/2},
\end{equation}
which is solved for any $Y$ of the form 
\begin{equation}\label{YZ}
Y=N_1^{-1/2}ZN_1^{-1/2},
\end{equation}
where $Z$ is an arbitrary square root of the expression $N_1^{1/2}N_2N^{1/2}$, i.e., 
\begin{equation}\label{ZY}YN_1Y=N_2\Leftrightarrow Y=N_1^{-1/2}ZN_1^{-1/2},\quad Z\in\sqrt{N_1^{1/2}N_2N_1^{1/2}}, \ N_1^{1/2}\in\sqrt{N_1}.\end{equation}

With this observation in mind, we proceed with the main result of this paper. For the sake of completeness, we formulate it in its full form.

\begin{theorem} [Solvability and particular solution: $u$ and $v$ version]\label{upmv}
Let $a\in\A_1$, $b\in\A_2$, be given in the manner that $\sigma(a)\cup\sigma(b)\subset\Lambda_{\pi/4}$, and that $\sigma(a)\cap\sigma(b)\neq\emptyset$. Let  $c\in\M$ be arbitrary, and denote by $\bar{c}$ the unique solution to the equation $a\bar{c}+\bar{c}b=c$.  Moreover, let $r:=a^{-1}\bar{c}b+ a\bar{c}b^{-1}$. The following statements are equivalent: 
\begin{itemize}
\item[(a)] The equation \eqref{sylm-} is solvable, i.e., there exists an $x\in\M$ such that $$ax-xb=c.$$
\item[(b)] There exist elements $u$ and $v$ in $\M$ which solve any two (all) of the following equations:
\begin{equation}\label{gensylf}
av+ub=\bar{c}.
\end{equation}
\begin{equation}\label{gensylvuf}
au+vb=\bar{c}+a^{-1}\bar{c}b+a\bar{c}b^{-1}.
\end{equation}
 \begin{equation}\label{u+vf}
u+v=a^{-1}cb^{-1}.
\end{equation}
 \begin{equation}\label{RajUchiuvf}
a^3v+a^2vb+ub^3+aub^2=0.
\end{equation}
\item[(c)] There exist elements $u$ and $v$ in $\M$ which solve any one of the equations enlisted in \eqref{gensylf}--\eqref{RajUchiuvf}, while their difference $q:=v-u$ solves:
\begin{equation}\label{v-u}
\mat{1_{\A_1}}{q}{0}{1_{\A_2}}\mat{a}{-\bar{c}}{0}{-b}\mat{1_{\A_1}}{q}{0}{1_{\A_2}}=\mat{a}{-\bar{c}-r}{0}{-b}.
\end{equation}
\end{itemize}
If the equation \eqref{sylm-} is solvable, then with respect to the mentioned parameters $u$ and $v$, the expressions
\begin{equation}\label{solutionsuvab}
x_v:=-(a^2vb^{-1}+av);\quad x_u:=a^{-1}ub^2+ub
\end{equation}
define one and the same particular solution to \eqref{sylm-}.
\end{theorem}
\begin{proof} $(a)\Longleftrightarrow(b)$: This follows from Corollary \ref{gensystemsol}, Lemma \ref{2of4}, and Theorem \ref{4of2}.

$(a)\Rightarrow(c):$ By  Corollary \ref{gensystemsol} and Theorem \ref{4of2}, the equation \eqref{sylm-} is solvable if and only if any two (all) the conditions enlisted in \eqref{gensylf}--\eqref{RajUchiuvf} are true. Then, from  \eqref{gensylf} and \eqref{gensylvuf} it follows that
$$N_0^-=\mat{a}{0}{0}{-b}=\mat{1_{\A_1}}{-u}{0}{1_{\A_2}}\mat{a}{-\bar{c}}{0}{-b}\mat{1_{\A_1}}{v}{0}{1_{\A_2}}$$
while simultaneously
$$N_0^-=\mat{a}{0}{0}{-b}=\mat{1_{\A_1}}{-v}{0}{1_{\A_2}}\mat{a}{-\bar{c}-r}{0}{-b}\mat{1_{\A_1}}{u}{0}{1_{\A_2}}.$$
Equating the two gives \eqref{v-u}. 

$(c)\Rightarrow(a):$ Assume that \eqref{v-u} is satisfied. Then as before we have
\begin{eqnarray}\label{mat-uv}
\begin{aligned}
&\mat{1_{\A_1}}{-u}{0}{1_{\A_2}}\mat{a}{-\bar{c}}{0}{-b}\mat{1_{\A_1}}{v}{0}{1_{\A_2}}=\\
=&\mat{1_{\A_1}}{-v}{0}{1_{\A_2}}\mat{a}{-\bar{c}-r}{0}{-b}\mat{1_{\A_1}}{u}{0}{1_{\A_2}}.
\end{aligned}
\end{eqnarray}
If \eqref{gensylf} is true, then the left-hand side of the latter is equal to $N_0^-$, implying that the right-hand side is also equal to $N_0^-$. However, this concludes that \eqref{gensylvuf} is also true, and by Lemma \ref{2of4}, the equation \eqref{sylm-} is solvable. The same reasoning holds if one assumes that \eqref{gensylvuf} is true: it follows from \eqref{mat-uv} that \eqref{gensylf} must also be true. 

Now assume that \eqref{u+vf} is true. Then $u=a^{-1}cb^{-1}-v$, and substituting this into $v-u=q$ implies that $q=2v-a^{-1}cb^{-1}$. Then from \eqref{v-u} we have
$$\begin{aligned}
&\mat{a}{-\bar{c}-r}{0}{-b}=\\
=&\mat{1_{\A_1}}{2v-a^{-1}cb^{-1}}{0}{1_{\A_2}}\mat{a}{-\bar{c}}{0}{-b}\mat{1_{\A_1}}{2v-a^{-1}cb^{-1}}{0}{1_{\A_2}}=\\
=&\mat{a}{-\bar{c}-2vb+a^{-1}c}{0}{-b}\mat{1_{\A_1}}{2v-a^{-1}cb^{-1}}{0}{1_{\A_2}}=\\
=&\mat{a}{2av-cb^{-1}-\bar{c}-2vb+a^{-1}c}{0}{-b},
\end{aligned}$$
and this equality is true if and only if
$$2av-2vb=cb^{-1}-a^{-1}c-r=cb^{-1}-a^{-1}c-a^{-1}\bar{c}b-a\bar{c}b^{-1}.$$
Writting $c$ as $a\bar{c}+\bar{c}b$ gives
$$2av-2vb=a\bar{c}b^{-1}+\bar{c}-\bar{c}-a^{-1}\bar{c}b-a^{-1}\bar{c}b-a\bar{c}b^{-1}=-2a^{-1}\bar{c}b,$$
i.e., \eqref{avvb} is true. By Lemma \ref{uvsyl} and Lemma \ref{2of4}, this is equivalent to the consistency of \eqref{sylm-}. 

Finally, assume that \eqref{RajUchiuvf} holds. Then we have
$$0=a(a^2v+ub^2)+(a^2v+ub^2)b\Leftrightarrow a^2v+ub^2=0\Leftrightarrow a^2v=-ub^2.$$
On the other hand, since $v=q+u$, where $q$ satisfies \eqref{v-u}, it follows that
$$a^2q+a^2u=-ub^2\Leftrightarrow a^2u+ub^2=-a^2q.$$
Now we have
$$a^2(u+v)+(u+v)b^2=2a^2u+a^2q+2ub^2+qb^2=a^2u+ub^2+qb^2=qb^2-a^2q.$$
Finally, from \eqref{v-u} it follows that
$$-\bar{c}-r=-\bar{c}-qb+aq\Leftrightarrow qb-aq=r\Leftrightarrow b^2q-a^2q=ar+rb=a^{-1}cb+acb^{-1}$$
where we used $b^2q-a^2q=(A+B)(B-A)(q)$ in terms of the multiplication operators $Aq=aq$ and $Bq=qb$. Therefore, equating the two expressions for $qb^2-aq^2$ gives
$$a^{-1}cb+acb^{-1}=a^2(u+v)+(u+v)b^2.$$
The latter is a regular Sylvester equation, and a direct verification shows that $a^{-1}cb^{-1}$ solves the equation, implying that \eqref{u+vf} is true, and the proof is complete. 

The expressions \eqref{solutionsuvab} follow immediately from \eqref{soluvab} in Theorem \ref{part}.
\end{proof}
Due to \eqref{YNY}--\eqref{ZY}, the equation \eqref{v-u} is indeed solvable for some $Y\in\mathcal{M}_0^{-1}$, provided via \eqref{YZ}.  Therefore, if among all such solution matrices $Y$ there are some of the form
$$Y_q=\mat{1_{\A_1}}{q}{0}{1_{\A_2}}$$
then \eqref{sylm-} is solvable, and one has a regular system $u+v=a^{-1}c^{-1}$, $v-u=q$, which defines a particular solution via \eqref{solutionsuvab}. Depending on how many different matrices $Y_q$ there are, one has different choices for a particular solution to \eqref{sylm-}.

\noindent\textbf{Data availability statement.} Data availability not applicable since no data sets were generated during this research.\\

\noindent\textbf{Declaration of conflict of interest.} The author declares there is no conflict of interest in publishing the results obtained in this research.\\

\noindent\textbf{Acknowledgements.} The author is supported by the Bulgarian Ministry of Education
and Science, Scientific Programme ``Enhancing the Research Capacity in Mathematical Sciences (PIKOM)",
No. DO1-67/05.05.2022, and by the Ministry of Science, Technological Development and Innovations, Republic of Serbia, grant No. 451-03-66/2026-03/200029.


\begin{thebibliography}{00}

\bibitem{Trap} Antoine, J.-P., Inoue, A., Trapani, C.: {\it Partial $*-$Algebras and their Operator Realizations}. Kluwer,
Dordrecht (2002)

\bibitem{R2} W. Arendt, F. R\"{a}biger and A. Sourour, \textit{Spectral properties of the operator equation $AX+XB=Y$}, Quart. J. Math. Oxford 2:45 (1994) 133--149.

\bibitem{BDI} Bellomonte, G., Djordjevi\'c, B., Ivkovi\'c, S.,  {\it On representations and topological aspects of positive
maps on non-unital quasi $*-$algebras}, Positivity 28(5), 66 (2024).

\bibitem{BIT} Bellomonte, G. Ivkovi\'c, S. Trapani, {Banach bimodule-valued positivemaps: inequalities and representations}, Banach J. Math. Anal. (2026) 20:12. https://doi.org/10.1007/s43037-025-00465-y

\bibitem{Bel} Bellomonte, G. Ivkovi\'c, S. Trapani, C., {\it GNS construction for positive $C^*-$valued sesquilinear maps on a quasi $*-$aglebra},  Mediterr. J. Math., 21 (2024) 166 (22 pp) (2024)
\bibitem{BL} A. Bezai, F. Lombarkia, {\it On the operator equation $AX-XB+XDX=C$},  Rend. Circ. Mat. Palermo, II. Ser 72, 4179--4187 (2023). https://doi.org/10.1007/s12215-023-00887-3


\bibitem{RBPR} R. Bhatia and P. Rosenthal, {\it How and why to solve the operator equation $AX-XB=Y$}, Bull. London Math. Soc. {29} (1997) 1--21.

\bibitem{RBMU} R. Bhatia and M. Uchiyama, {\it The operator equation $\sum_{i=0}^n A^{n-i}XB^i=Y$}, Expo. Math. 27 (2009) 251--255.

\bibitem{JB} J. Bra\v ci\v c, {\it Local commutants and ultrainvariant subspaces}, J. Math. Anal. Appl. 506 (2022) 125693. https://doi.org/10.1016/j.jmaa.2021.125693

\bibitem{ECLHDS} E. K.-W. Chu, L. Ho, D. B. Szyld and J. Zhou, {\it Numerical solution of singular Sylvester equations}, J. Comput. Appl. Math. 436 (2024) 115426 https://doi.org/10.1016/j.cam.2023.115426

\bibitem{DIN} N. \v C. Din\v ci\'c, {\it Solving the Sylvester equation $AX-XB = C$ when $\sigma(A)\cap \sigma(B) \neq\emptyset$}, Electron. J. Linear Algebra 35 (2019) 1--23.


\bibitem{BD1} B. D. Djordjevi\'c, {\it Operator algebra generated by an element from the module $\mathcal{B}(V_1,V_2)$}, Complex Analysis Operator Theory (2019)
https://doi.org/10.1007/s11785-019-00899-x.

\bibitem{BD2} B. D. Djordjevi\'c, {\it Singular Sylvester equation in Banach spaces and its applications: Fredholm theory approach} 622 (2021) 189--214. https://doi.org/10.1016/j.laa.2021.03.035

\bibitem{BDD2} B. D. Djordjevi\'c, {\it The equation $AX-XB=C$ without a unique solution: the ambiguity which benefits applications}, Topics in Operator Theory,  Zb. Rad. MISANU 20:28 (2022) 395--442
http://elib.mi.sanu.ac.rs/files/journals/zr/28/zrn28p395-442.pdf

\bibitem{BDND1} B. D. Djordjevi\'c and  N. \v C. Din\v ci\'c, {\it Classification and approximation of solutions to Sylvester matrix equation}, Filomat 13:33 (2019) 4261--4280
https://doi.org/10.2298/FIL1913261D

\bibitem{BDZG} B. D. Djordjevi\'c and  Z. Lj. Golubovi\'c, {\it Summation of hyperharmonic series in Banach algebras and Banach
bimodules}, Filomat 40:2 (2026), 583–600. https://doi.org/10.2298/FIL2602583D

\bibitem{MD} M. P. Drazin, {\it On a result of J. J. Sylvester}, Linear Algebra Appl. 505 (2016) 361--366 http://dx.doi.org/10.1016/j.laa.2016.05.007

\bibitem{DS} Dunford N., Schwartz J. T., {\it Linear Operators Part I: General Theory}, Wiley-Interscience 1988.

\bibitem{FRG} F. R. Gantmacher, {\it The Theory of Matrices}, Chelsea, New York 1959.

\bibitem{GW} F. Gerrish and A. J. B. Ward, {\it Sylvester's Matrix Equation and Roth's Removal Rule}, The Mathematical Gazette 82:495 (1998) 423--430.

\bibitem{MGLD} M. C. Gouveia and L. D. Dinis, {\it On the solutions of Sylvester, Lyapunov and Stein equations over arbitrary rings}, Int. J. Pur. Appl. Math. 24:1 (2005) 131--137.

\bibitem{REH} R. E. Hartwig, {\it Roth's removal rule revisited}, Linear Algebra Appl. 48 (1983) 91--115.

\bibitem{HC} Q. Hu and D. Cheng, {\it The polynomial solution to the Sylvester matrix equation}, Linear Algebra Appl. 172 (1992) 283--131.

\bibitem{SGL} S.-G. Lee and Q.-P. Vu, {\it Simultaneous solutions of operator Sylvester equations}, Studia Mathematica 222 (1) (2014) 87--96. DOI: 10.4064/sm222-1-6

\bibitem{ZMREMMFM} Z. Mousavi, R. Eskandari, M. S. Moslehian, F. Mirzapour, {\it Operator equations $AX+YB=C$ and $AXA^*+BYB^*=C$ in Hilbert $C^*-$modules}, Linear Algebra Appl. 517 (2017) 85–-98. http://dx.doi.org/10.1016/j.laa.2016.12.001

\bibitem{VM} V. M\"uller, {\it Spectral Theory of Linear Operators}, Birkh\"auser (2007).

\bibitem{HM} H. Mustafayev,  {\it Intertwining Conditions for Two Isometries on Banach Spaces}, Complex Anal. Oper. Theory 17, 131 (2023). https://doi.org/10.1007/s11785-023-01436-7

\bibitem{RW} W. E. Roth, {\it The Equations $AX-YB=C$ and $AX-XB=C$ in Matrices}, Proc. Amer. Math. Soc. 3 (1952) 392--396.

\bibitem{AS} A. Sasane, {\it The Sylvester equation in Banach algebras}, Linear Algebra Appl. {631} (2021) 1--9 https://doi.org/10.1016/j.laa.2021.08.015

\bibitem{ASMD} A. Shirilord and M. Dehghan, {\it Iterative method for constrained systems of conjugate transpose matrix
equations}, Appl. Numer. Math. (2024) https://doi.org/10.1016/j.apnum.2024.01.016

\bibitem{JJS} J. J. Sylvester, {\it Sur l'equation en matrices $px=xq$}, C. R. Acad. Sci. Paris, 99 (1884) 67--71 and 115--116.


\bibitem{HWXSJH} H. Wang, X. Sun and J. Huang, {\it On the solvability of generalized Sylvester operator equations} Operators Matrices 16:3 (2022) , 697--708 dx.doi.org/10.7153/oam-2022-16-51
\end{thebibliography}
\end{document}